\documentclass[11pt,a4paper]{article}
\usepackage{a4wide,amsmath,amsthm,amssymb,amsfonts,color}
\usepackage{mathtools}
\usepackage{hyperref} 
\mathtoolsset{showonlyrefs}
\usepackage[all]{xy}
\usepackage[utf8]{inputenc}
\usepackage{enumitem,MnSymbol,authblk}

\usepackage{tikz,tikz-3dplot}
\usepackage{mathrsfs}
\usepackage{cancel}

\newtheorem{thm}{Theorem}[section]
\newtheorem{prop}[thm]{Proposition}
\newtheorem{lma}[thm]{Lemma}

\newtheorem{dfn}[thm]{Definition}

\theoremstyle{remark}
\newtheorem{rmkk}[thm]{Remark}
\newtheorem{exe}[thm]{Example}
\newenvironment{rmk}{\begin{rmkk}\rm}{\qee\end{rmkk}}

\newcommand{\qee}{\mbox{\hspace{0.2mm}}\hfill$\triangle$}
\newcommand{\Z}{\mathbb{Z}}
\newcommand{\R}{\mathbb{R}}
\newcommand{\C}{\mathbb{C}}
\newcommand{\Q}{\mathbb{Q}}
\newcommand{\Pj}{\mathbb{P}}

\newcommand{\N}{\mathbb{N}}
\newcommand{\cO}{\mathcal{O}}

\newcommand{\codim}{\operatorname{codim}}

\newcommand{\Pic}{\operatorname{Pic}}

\title{\bf \large DEFORMATION OF PAIRS AND NOETHER-LEFSCHETZ  \\[5pt] LOCI IN TORIC VARIETIES }

\author[$^\S$]{Ugo Bruzzo\footnote{On leave of absence from SISSA (International School for Advanced Studies), Trieste, Italy}}
\author[$\ddag$]{William D. Montoya}
\affil[$^\S$]{Departamento de Matem\'atica, Universidad Federal da Para\'iba, \par\vskip-3pt  Campus I, 58051-900, Jo\~ao Pessoa, PB, Brazil}
\affil[$^\S$]{INFN (Istituto Nazionale di Fisica Nucleare), Sezione di Trieste}
\affil[$^\S$]{IGAP (Institute for Geometry and Physics), Trieste}
\affil[$^\S$]{Arnold-Regge Center for Algebra, Geometry \par\vskip-3pt and Theoretical Physics, Torino}
 \affil[$\ddag$]{Instituto de Matem\'atica,  Estat\'istica e Computa\c c\~ao  Cient\'ifica, \par\vskip-3pt Universidad Estadual de Campinas, Rua S\'ergio Buarque \par\vskip-3pt de Holanda 651, 
 13083-859, Campinas, SP, Brazil}
\date{}
\begin{document}
\maketitle

\begin{abstract} We continue our    study of the Noether-Lefschetz loci in toric varieties  and investigate deformation of pairs $(V,X)$ where $V$ is a complete intersection subvariety and $X$ a quasi-smooth hypersurface  in a simplicial projective toric variety $\Pj_{\Sigma}^{2k+1}$, with $V\subset X$. Under some assumptions, we prove that the cohomological class $\lambda_V\in H^{k,k}(X)$ associated to $V$ remains of type $(k,k)$ under an infinitesimal deformation if and only if $V$ remains algebraic.
Actually we prove that locally the Noether-Lefschetz locus is an irreducible component of a suitable Hilbert scheme. This generalizes Theorem 4.2 in our previous work   \cite{Papelito3} and the main theorem proved by Dan in \cite{Dan}.  
\end{abstract}
    
 \let\svthefootnote\thefootnote
\let\thefootnote\relax\footnote{
\hskip-2\parindent 
Date:  \today  
Email: {\tt  bruzzo@sissa.it, montoya@unicamp.br}
}
\addtocounter{footnote}{-1}\let\thefootnote\svthefootnote

\newpage
\section{Introduction}
In this short note we  continue our    study of the Noether-Lefschetz loci in toric varieties  and investigate the deformation of pairs $(V,X)$ where $V$ is a $k$-dimensional  complete intersection subvariety and $X$ a quasi-smooth ample hypersurface  in a simplicial projective toric variety $\Pj_{\Sigma}^{2k+1}$ of odd dimension $2k+1\ge  3$,  with $V\subset X$. 
 We assume that the local 
Noether-Lefschetz locus $NL_{\lambda_{V},U}^{k,\beta}$, also called ``Hodge locus'' in the literature when $\Pj_{\Sigma}^{2k+1}$ is a projective space, 
as defined in Section \ref{main},  is not empty (a condition for this to happen is for instance given in Lemma 3.7 of  \cite{BruzzoGrassi}).   Here $\lambda_V$ is the  cohomology  class  of $V$, and $\beta$ is the class
of $X$ in $\Pic(\Pj_{\Sigma}^{2k+1})$.
 Then every irreducible component $L$  of the full Noether-Lefschetz locus $NL_\beta$, namely the locus  in the linear system $\vert\beta\vert$ of the points corresponding
to quasi-smooth  hypersurfaces   whose $(k,k)$-cohomology does not come entirely from the ambient $\Pj_{\Sigma}^{2k+1}$,
  is locally the Hodge  locus \cite{Papelito2}.
In other terms, there exists an open subset $U$ of $\vert\beta\vert$ such that $L\cap U= NL_{\lambda_{V},U}^{k,\beta}$.

Moreover, under the further
 assumption that $\beta$   satisfies $\beta= q\,\eta+\beta'$ $(n\in\N)$, where $q\in \Q_{>0}$, $\eta$ is a primitive ample class in $\Pj_{\Sigma}^{2k+1}$,
 and $\beta'$ is a nef Cartier class, 
  if $X$ cointains a $k$-dimensional complete intersection subvariety with $\deg_\eta V <  q$,  
we will show that under  infinitesimal deformations, $V$ is algebraic if and only if its associated cohomology class $\lambda_V$ is a $(k,k)$ class. 

This extends the work of Dan in \cite{Dan} and the last result of \cite{Papelito3} (Theorem 4.2) for toric varieties with higher Picard rank
(there the Picard number was assumed to be one, and moreover, the result is asymptotic).

\textbf{Acknowledgement.}
We thank Ananyo Dan for  interesting correspondence about the variational Hodge conjecture in   projective spaces.  The first
author's research is partially supported by the CNPq ``Bolsa de Produtividade em Pesquisa'' 313333/2020, by
the PRIN project ``Geometria delle variet\`a algebriche'' and by GNSAGA-INdAM. The second author acknowledges support from FAPESP postdoctoral grant no. 2019/23499-7.

\section{Infinitesimal variation of the Hodge structure}

According to Batyrev and Cox in \cite{BatyrevCox},  the cohomology of hypersurfaces in  projective simplicial toric varieties has a pure Hodge structure. In this section, we introduce its infinitesimal variation following the notions introduced by Carlson, Green, Griffiths and Harris in  \cite{CarlsonGreenGriffithsHarris}.

\begin{dfn} A polarized Hodge structure of weight $n$, denoted by $\{H_{\Z}, H^{p,q} , Q\}$, is a Hodge structure together with a bilinear form $Q:H_{\Z}\times H_{\Z}\rightarrow \Z$ satisfying
$$
\begin{array}{ll} 
Q(\psi,\phi)=(-1)^nQ(\phi,\psi)& \\[3pt]
Q(\psi,\phi)=0& \psi\in H^{p,q},\phi\in H^{p',q'}\, and \, p\neq q'\\[3pt]
i^{p-q}Q(\psi,\bar{\psi})>0 & 0\neq \psi\in H^{p,q}
\end{array} 
$$

\end{dfn}

\begin{dfn}  An infinitesimal variation of Hodge structure $\{H_{\Z}, Hp,q, Q, T, \delta \}$
 is given by a polarized Hodge structure 
together with a vector space $T$ and linear map  
$$\delta: T\rightarrow \bigoplus_{1\leq p \leq n} Hom(H^{p,q}, H^{p-1,q+1}) $$
that satisfies the two conditions:
$$\delta(\xi_1)\delta(\xi_2)=\delta(\xi_2)\delta(\xi_1),\quad \xi_1,\xi_2\in T$$
$$Q(\delta(\xi)\phi,\psi) + Q(\phi,\delta(\xi)\psi)=0 \quad \text{for} \  \xi\in T \ \text{and} \ \phi\in F^p,\ \psi\in F^{n-p+1}.$$
\end{dfn}
Here $F^\bullet$ is the filtration of $H^n$ given by
$$ F^ p =  \bigoplus_{i=0}^p H^{n-i,i}.$$

For a quasi-smooth hypersurface $X$ in a simplicial projective  toric variety, $\delta$ is the morphism associated via   tensor-hom adjuction to $\gamma=\sum_p \gamma_p$, where 
$\gamma_p: T_X\mathcal{M}_{\beta} \otimes H_{\text{prim}}^{p,d-1-p}(X)\rightarrow  H_{\text{prim}}^{p,d-1-p}(X)$ is the natural multiplication map; for more details  see   3.3 in \cite{BruzzoGrassi}. 
Given an infinitesimal variation of Hodge structure of weight $2k$, there is an invariant associated to $\gamma\in H^{k,k}_{\Z}$.

\begin{dfn} The {\em  third invariant} associated to  $\gamma\in H^{k,k}_{\Z}$ is 
$$H^{k,k}(-\gamma):=\{\psi\in H^{k,k} \mid <\delta^0(\xi)\psi, \gamma>=0 \, \forall \xi \in T\} .$$

\end{dfn}

Let us assume $\gamma$ is the primitive part of the   class  of $k$-codimensional algebraic cycle $V=\sum_i n_i V_i$ in $X$ with support $\sigma(V)$.  Let $I_{\sigma(V)}$ be the ideal associated to $\sigma(V)$  and  denote by $H^k(\Omega_X^{k}(-V))$  the image of the composed map 
$$H^k(X,\Omega^{k}_X\otimes I_{\sigma(V)})\rightarrow  H^k(X,\Omega^k_X)\rightarrow H^{k}_{\text{prim}}(\Omega_{X}^{k}).$$

One has the following fact (\cite{GriffithsHarris}, Observation 4.a.4). 
\begin{lma}\label{obs} $H^k(\Omega_X^{k}(-V))\subseteq H^{k,k}(-\gamma) $.
\end{lma}
This is the result that we shall need later on.

\section{A dual Macaulay theorem in toric varieties}
We review a  theorem of Macaulay theorem proved in \cite{Papelito3}.   We refer the reader to that paper for all proofs of results in this section.

The Cox ring $S$ of a complete simplicial toric variety $\Pj_\Sigma$  is graded over the effective classes in the class group $\operatorname{Cl}(\Pj_\Sigma)$ 
 $$S = \sum_{\alpha \in \operatorname{Cl}(\Pj_\Sigma)} S^\alpha, \qquad
S^\alpha =  H^0(\Pj_\Sigma,\cO_{\Pj_\Sigma}(\alpha))$$
 (see e.g.~\cite{CoxHom}).
 As $\cO_{\Pj_\Sigma}(\alpha)$ is coherent and $\Pj_\Sigma$ is complete,  each $S^\alpha$ is finite-dimensional over $\C$; in particular, $S^0\simeq \C$. 
For every effective $N\in\operatorname{Cl}(\Pj_\Sigma)$, the set of classes $\alpha\in\operatorname{Cl}(\Pj_\Sigma)$ such that $N-\alpha$ is effective
 is finite.

 We shall give a definition of {\em Cox-Gorenstein ideal} of the Cox rings which generalizes to toric varieties the definition
 given by Otwinowska in \cite{Ania} for projective spaces. 
Let  $B\subset S$ be the  irrelevant ideal, and for a graded ideal $I\subset B$,
denote by $V_{\mathbb T}(I)$ the corresponding closed subscheme of $\Pj_\Sigma$.
 
 \begin{dfn} \label{lambdator} 
 A graded ideal $I$ of $S$ contained in $B$ is said to be
 a Cox-Gorentstein ideal of socle degree  $N\in\operatorname{Cl}(\Pj_\Sigma)$  if
 \begin{enumerate}\itemsep=-2pt
\item there exists a $\C$-linear form
$\Lambda\in (S^N)^\vee$ such that for all $\alpha \in\operatorname{Cl}(\Pj_\Sigma)$
\begin{equation}\label{eqlambda} I^\alpha  =\{f\in S^\alpha \,\vert\, \Lambda(fg) = 0 \ \mbox{for all} \ g\in S^{N-\alpha }\}; \end{equation}
\item $V_{\mathbb T}(I)=\emptyset$.
\end{enumerate} 
\end{dfn}

\begin{prop}
Let $R=S/I$. If $I$ is Cox-Gorenstein then  \begin{enumerate}\itemsep=-2pt
\item $\dim_\C R^N = 1$; \item
the natural bilinear morphism (called ``Poincar\'e duality'') \begin{equation}\label{pair}R^\alpha \times R^{N-\alpha } \to R^N\simeq \C \end{equation} is nondegenerate
whenever $\alpha$ and $N-\alpha$ are effective.   \end{enumerate} 
\end{prop}

We shall need to use a {{\em Euler form} on $\Pj_\Sigma$ \cite{BatyrevCox,cox-res,ccd}.} We  denote by $M$ the dual lattice of the lattice $N$ which contains the fan $\Sigma$, i.e., $\Sigma\subset N\otimes  \R$.  Let $d=\dim \Pj_\Sigma$.

\begin{dfn}\label{form}Fix an integer basis  $u_1,\dots u_d$ for the lattice $M$. Then given a subset $\iota=\{i_1,\dots, i_d\}\subset \{1,\dots,\#\Sigma(1) \} $, where $\#\Sigma(1)$ is the number of rays in $\Sigma$, we define
$$\det(e_{\iota}):=\det\big(<u_j,e_{i_{h}}>_{1\leq j,h\leq d} \big) ;$$
moreover, $dx_{\iota}=dx_{i_1}\wedge\cdots \wedge dx_{i_d} $ and $\hat{x}_{\iota}=\Pi_{i\nin\iota}x_{\iota}$.
\end{dfn}

\begin{dfn} A Euler form on $\Pj_\Sigma$ is a {Zariski} 
$d$-form $\Omega_0 $   defined as 
$$\Omega_0:=\sum_{|\iota|=d}\det(e_{\iota})\hat{x}_{\iota}dx_{\iota}$$
where the sum is over all subsets $\iota\subset \{1,\dots ,\#\Sigma(1)\}$ with $d$ elements. 
\label{euler}
\end{dfn}
For more details about these definitions  see \cite{BatyrevCox}.

Let $f_0,\dots,f_d$ be homogeneous polynomials with $ \deg(f_i) = {\alpha_i}$, where  each $\alpha_i$ is ample, and let $ N = \sum_i\alpha_i-\beta_0$; here
$\beta_0$ is the anticanonical class of $\Pj_\Sigma$. Assume that the $f_i$ have no common zeroes in $\Pj_\Sigma$, i.e., $V_{\mathbb T}(f_0,\dots,f_d)=\emptyset$. For each $G \in S^N$ one can define a meromorphic 
$d$-form $\xi_G$ on $\Pj_\Sigma$ by letting
$$ \xi_G = \frac{G\,\Omega_0}{f_0\cdots f_d}$$
where $\Omega_0$ is a Euler form on $\Pj_\Sigma$. The form $\xi_G$ determines a class in $H^d(\Pj_\Sigma,\omega)$,
where $\omega$ is the canonical sheaf of $\Pj_\Sigma$, i.e., the sheaf of Zariski $d$-forms on $\Pj_\Sigma$, and  
the trace morphism  $\operatorname{Tr}_{\Pj_\Sigma}\colon H^d(\Pj_\Sigma,\omega)\to\C$ associates a complex number to $G$. We  
define $\Lambda\in (S^N)^\vee$ as 
\begin{equation}\label{deflambdator} \Lambda(G) = \operatorname{Tr}_{\Pj_\Sigma}([\xi_G])\in \C.\end{equation}
 
We can now state a toric version of Macaulay's theorem.

\begin{thm}  
The linear map defined in Eq.~\eqref{deflambdator} 
satisfies the condition in Definition \ref{lambdator}. Therefore, the ideal $I=(f_0,\dots,f_d)$ is a Cox-Gorenstein ideal of socle degree $N=\sum_i\deg (f_i)-\beta_0$.
\label{toricmac}
\end{thm}

Examples of   Cox-Gorenstein ideals may be given in terms of {\em toric Jacobian ideals.}
For every ray $\rho\in\Sigma(1)$  denote by $v_\rho$ its rational generator,
and by $x_\rho$ the corresponding variable in the Cox ring. Recall that $d$ is the dimension of the toric variety $\Pj_\Sigma$, while we denote
by $r=\#\Sigma(1)$ the number of rays. 
Given $f\in S^\beta$ one defines its {\em toric Jacobian ideal}   as 
$$J_0(f) = \left( x_{\rho_1} \frac{\partial f }{\partial  x_{\rho_1}}, \dots,  x_{\rho_r} \frac{\partial f }{\partial  x_{\rho_r}} \right).$$

We recall from
\cite{BatyrevCox} the   definition of nondegenerate hypersurface and some properties (Def.~4.13 and Prop.~4.15).

\begin{dfn} Let $f\in S^\beta$, with $\beta$ an ample Cartier class. The associated hypersurface $X_f$ is nondegenerate if for all $\sigma\in\Sigma$ the affine hypersurface $X_f\cap O(\sigma)$ is a smooth codimension one subvariety of the orbit $O(\sigma)$ of the action of the torus
$\mathbb T^d$. 
\end{dfn}

\begin{prop} \begin{enumerate} \itemsep=-3pt\item Every nondegenerate hypersurface is quasi-smooth.
\item If $f$ is generic then $X_f$ is nondegenerate.
\end{enumerate} 
\end{prop}

We collect here, with some changes in the terminology, some results that are already 
 contained in Prop.~5.3 of \cite{cox-res}.
 
\begin{prop} Let $f\in S^\beta$, and let $\{\rho_1,\dots,\rho_d\}\subset\Sigma(1)$ be such that
$v_{\rho_1},\dots,v_{\rho_d}$ are linearly independent.  \begin{enumerate}  \item The toric Jacobian ideal of $f$ coincides with the ideal 
$$ \left( f, x_{\rho_1} \frac{\partial f }{\partial  x_{\rho_1}}, \dots,  x_{\rho_d} \frac{\partial f }{\partial  x_{\rho_d}} \right).$$
\item The following conditions are equivalent:
\begin{enumerate} \item $f$ is nondegenerate; \item the polynomials $x_{\rho_i} \frac{\partial f }{\partial  x_{\rho_i}}$, $i=1,\dots,r$, do not vanish simultaneously on $X_f$; \item
the polynomials  $f$ and $x_{\rho_i} \frac{\partial f }{\partial  x_{\rho_i}}$, $i=1,\dots,d$, do not vanish simultaneously on $X_f$.
\end{enumerate}
\item If   $\beta$ is ample and $f$ is nondegenerate, then $J_0(f)$ is a   Cox-Gorenstein ideal of socle degree
$N = (d+1)\beta-\beta_0$, where $\beta_0$ is the anticanonical class of $\Pj_\Sigma^d$.
\end{enumerate}
\end{prop}

\section{The tangent space to the Noether-Lefschetz locus}
From now we assume $d=2k+1$. Let $N=(k+1)\beta-\beta_0$ and let $J_0(f)$ be the toric Jacobian ideal associated to $f$, which is Cox-Gorenstein of  socle degree $2N+\beta_0$. Then there is a perfect pairing $R_0^{\alpha}\times R_0^{2N+\beta_0-\alpha}\rightarrow R_0^{2N+\beta_0}$  for  $\alpha\leq 2N+\beta_0$. Let us denote by $T'_0$  the subspace of $R^N_0$ which is the kernel of the multiplication  map $\cdot x_1,\dots, x_r P: R_0^{N}\rightarrow R_0^{2N+\beta_0}$ and by $T_0$  {its} lift in $S^{N}$, where $P$ is a preimage of $\gamma$ under the   natural map
$$
\begin{array}{ccccl}
S^N&\to & S^N/J^N & \xrightarrow{\sim} & H^{k,k}_{\text{prim}}(X) \\[5pt]
P&\mapsto& \bar{P} &\mapsto  &\gamma

\end{array}
$$

\begin{dfn}  $T\subset S$ be  the   $Cl(\Sigma)$-graded module such that $T^{\alpha}$ is the largest subspace where $T^{\alpha}\otimes S^{N-\alpha} $ is contained in $T_0$ for $\alpha\leq N$, $T^N=T_0$ and $E^{N+\alpha}=T_0\otimes S^{\alpha}$ for $\alpha\geq 0$.
 
\end{dfn}

\begin{rmk} Note that $T$ is a Cox-Gorentein ideal with socle degree $N$. \end{rmk}

Actually  $T^{\beta}$ is the tangent space of the local Noether-Lefschetz locus at $f$:

\begin{lma} $T_{f}NL_{\lambda,\beta}\cong T^{\beta}$.
\end{lma}

 \begin{proof} A superimposed bar will denote the class in $R=S/J$ of an element in $S$. 
Now, $H\in T^{\beta}$ if and only if $\bar{H}\otimes R^{N-\beta}$  is contained in $T'_0$, which is equivalent to $$\overline{x_0...x_r}\bar{P}\bar{H}\otimes R^{N-\beta}=0\quad\text{in}\quad R^{N+\beta_0};$$ using Poincar\'e duality that means $\overline{x_0...x_r}\bar{P}\bar{H}=0$ in $R^{N+\beta+\beta_0}$ and equivalently $\bar{P}\bar{H}=0$ in $R^{N+\beta}$ if and only if $H\in T_{f}NL_{\lambda,\beta}$. (see Theorem 6.2 in \cite{Papelito1}). 
\end{proof}

Let us suppose that $V$ is the zero locus of $<A_1,\dots, A_{k+1}>$ and since $V\subset X_f$ there exist $K_1,\dots, K_{k+1}$ polynomials of degree $\beta-\deg(A_i)$ such that $f=A_1K_1+\dots +A_{k+1}K_{k+1}$. Let $I=<A_1, \dots, A_{k+1},K_1,\dots K_{k+1}>$.

\begin{prop}$ T^{\alpha}=I^{\alpha}$ for  $\alpha \leq N$.

\end{prop}

\begin{proof} 
Let $W_1$ be the zero locus of $<K_1, A_2,\cdots, A_{k+1}>$. Since $V\cup W_1$ is equal to $X_f\cap \{A_2=\dots= A_{k+1}=0\}$ then $\lambda_{V}$ is equal to $\lambda_{W_1}$ in the primitive cohomology. Now, let us denote $W_2$ the zero locus of $K_1, \dots,K_{k+1}$  then, as before,  $[\lambda_{V}]_{\text{prim}}=[a\lambda_{W_2}]_{\text{prim}}$ ($a\in \Z$). By Lemma \ref{obs} we have $<A_1,\dots, A_{k+1}, K_1\dots K_{k+1}>\subset T$. Since $X$ is quasi-smooth the ideal 
$<A_1,\dots A_{k+1}, K_1,\dots K_{k+1} >$ is Cox-Gorenstein with socle degree $N$, the socle degree of $T$, so tha$I$ and $T$ coincide in degree $\alpha\leq N$ . 
\end{proof}

\section{Main theorem}\label{main}

 In this section we prove our main result. We start by recalling the construction of the local Noether-Lefschetz locus \cite{Papelito1}.
Given an ample class $\beta$  in $\Pic(\Pj^{2k+1}_{\Sigma})$, let $$\mathcal{U}_{\beta}\subset \Pj(H^0(\Pj^{2k+1}_{\Sigma}),\cO_{\Pj^{2k+1}_{\Sigma}}(\beta))) $$ be the open subset parameterizing   quasi-smooth hypersurfaces and let $\pi: \chi_{\beta}\to  \mathcal{U}_{\beta}$ be the tautological family. One considers
the local system $\mathcal{H}^{2k} = R^{2k}\pi_{\star}\C \otimes \cO_{\mathcal{U}_{\beta}}$ over $\mathcal{U}_{\beta}$.
 
 If $f\in \mathcal{U}_{\beta}$, assume it determines a nonzero class
$\lambda_f\in H^{k,k}(X_f,\Q)/i^*(H^{k,k}(\Pj^{2k+1}_{\Sigma},\Q))$ and let $U\subset  \mathcal{U}_{\beta}$ be a contractible open subset around $f$.  Finally, let $\lambda\in \mathcal{H}^{2k}(U)$ be the section defined by $\lambda_f$ and let $\bar{\lambda}$ be  its image in $(\mathcal{H}^{2k}/F^k\mathcal{H}^{2k})(U)$, where $$ F^k\mathcal{H}^{2k} =\mathcal{H}^{2k,0}\oplus \mathcal{H}^{2k-1,1}\oplus \dots \oplus \mathcal{H}^{k,k}.$$

\begin{dfn}[Local Noether-Lefschetz Locus] $NL^{k,\beta}_{\lambda,U} = \{G\in U \mid \bar{\lambda}_{G}=0\}$. \label{defNL}\end{dfn}

 Let $\eta$ be a polarization for $\Pj_{\Sigma}^{2k+1}$, that we assume to be primitive in the Picard group.  Given the Hilbert polynomial $P$ of a subscheme $V$, computed with respect to $\eta$,  we denote by $\operatorname{Hilb}_P$ the   Hilbert scheme  {of closed subschemes of $\Pj_{\Sigma}^{2k+1}$ with Hilbert polynomial $P$.} We denote by $Q$ the Hilbert polynomial of   quasi-smooth hypersurface in $\Pj_{\Sigma}^{2k+1}$  whose class in the Picard group is $\beta$. The flag Hilbert scheme $\operatorname{Hilb}_{P,Q}$ parametrizes all pairs $(V,X)$ where $V\in \operatorname{Hilb}_P$ and $X$  is a  quasi-smooth hypersurface in $\Pj_{\Sigma}^{2k+1}$  of  class $\beta$ containing $V$.  Let $\operatorname{pr}_1$ be the projection to the first component and $\operatorname{pr}_2: \operatorname{Hilb}_{P,Q}\rightarrow \mathcal{U}_{\beta} $ the natural projection to the open set which parametrizes  quasi-smooth hypersurfaces in $\Pj_{\Sigma}^{2k+1}$. Note that $\operatorname{pr}_1 (\operatorname{Hilb}_{P,Q})$ is irreducible, so that there exists a unique component in  $\operatorname{Hilb}_{P,Q}$ such that $\operatorname{pr}_1(\operatorname{Hilb}_{P,Q})$ coincides with the parameter space for {complete intersection subschemes  in $\Pj^{2k+1}_{\Sigma}$.

For $Z$ a $d$-dimensional closed subvariety of  $\Pj^{2k+1}_{\Sigma}$ we define its degree as $\deg_\eta Z = [Z] \cdot \eta^d$.

\begin{lma} \label{inNL}    If $X$ is an ample Cartier hypersurface in $\Pj^{2k+1}_{\Sigma}$ whose class in $\Pic(\Pj^{2k+1}_{\Sigma})$ satisfies
$\beta = q\,\eta +\beta'$, where $q\in\Q_{>0}$ and $\beta'$ is a nef Cartier class, and $V$ is a complete intersection $k$-dimensional subvariety contained in $X$, then $\deg_\eta V \ge   q$.
\end{lma}

\begin{proof} We shall denote by $(Z)$ the  class  in $A_d(\Pj^{2k+1}_{\Sigma})$ of a  $d$-dimensional closed subvariety $Z$ of $\Pic(\Pj^{2k+1}_{\Sigma})$, and by $[Z]$ its  class in $A^{2k+1-d}(\Pj^{2k+1}_{\Sigma})$. By hypothesis we have $V = X \cap W$ for a closed $(k+1)$-dimensional subvariety $W$ of
$\Pj^{2k+1}_{\Sigma}$. Thus we have
\begin{eqnarray} \deg_\eta V &=&  \langle \eta^k,(W)\cap [X]\rangle = \langle\eta^k\cup[X] , [W]\rangle  =  \langle \eta^k\cup(q\,\eta+\beta') , [W]\rangle  \\
&=& q \deg_\eta W +   \langle \eta^k\cup  \beta' , [W]\rangle\   \ge  \ q  +  \langle \eta^k\cup  \beta' , [W]\rangle. \end{eqnarray}
Since $\beta'$ is nef we have $\langle \eta^k\cup  \beta' , [W]\rangle \ge 0$, hence the claim follows (note that $\deg_\eta W$ is a positive integer). 
\end{proof}

Now we state and prove the main result of this paper.

\begin{thm}  Assume that   $\beta$ is as in the Lemma. Let $V$ be a quasi-smooth intersection in $\Pj^{2k+1}_{\Sigma}$ of codimension $k+1$ and let $X$ be a quasi-smooth hyper-surface of class $\beta$ containing $V$ such that  $\deg_\eta V <  q$. 
Assume also that the Noether-Lefschetz locus $NL_{\lambda_{V}\beta}$  is nonempty.
Then,

\begin{center}
    $\lambda_V$ deforms  to a  $(k,k)$ class if and only if $\lambda_{[V]}$  deforms to an algebraic cycle.

\end{center} 
In particular, $NL_{\lambda_{V},U}^{k,\beta}$ is isomorphic to an irreducible component of  {$\operatorname{pr}_2(\operatorname{Hilb}_{P,Q})$},
where $P$, $Q$ are the Hilbert polynomials of $V$ and $X$, respectively.

\end{thm}

\begin{proof} 
By the assumption on the degree of $V$,  one has $\operatorname{pr}_2(\operatorname{Hilb}_{P,Q})\subset NL_{\lambda_{V},U}^{k,\beta}$.
Then,
 $$\codim_U \operatorname{pr}_2(\operatorname{Hilb}_{P,Q}) \geq \codim_U NL_{\lambda_{V},U}^{k,\beta} \geq \codim_{T_XU} T_X NL_{\lambda_{V},U}^{k,\beta }.$$

On the other hand, keeping in mind that $T^{\beta}=I^{\beta} \subset I_V^{\beta}$, we have a natural map $\phi$ from $T_{\beta}$ to $\operatorname{Hilb}_{P,Q}$, which sends  a homogeneous polynomial of degree $\beta$ to its zero locus. One has  $\overline{\operatorname{Im}( \phi)}\subset \overline{\operatorname{pr}_2(\operatorname{Hilb}_{P,Q})}$ and since the zero locus is invariant under the torus action, $\dim T^{\beta}>\dim \overline{\operatorname{Im}(\phi)}$. Hence,
$$\codim \operatorname{pr}_2(\operatorname{Hilb}_{P,Q}) \leq  \codim \overline{\operatorname{Im }(\phi)} \leq \codim T^{\beta}=\codim T_X NL_{\lambda_{V},U}^{k,\beta}. $$
So $\operatorname{pr}_2(\operatorname{Hilb}_{P,Q})$ and $ NL_{\lambda_{V},U}^{k,\beta}$ have the same dimension, which implies the claim.
\end{proof}


\begin{thebibliography}{10}

\bibitem{BatyrevCox}
{\sc V.~V. Batyrev and D.~A. Cox}, {\em On the {H}odge structure of projective
  hypersurfaces in toric varieties}, Duke Math. J., 75 (1994), pp.~293--338.

\bibitem{BruzzoGrassi}
{\sc U.~Bruzzo and A.~Grassi}, {\em Picard group of hypersurfaces in toric
  3-folds}, Internat. J. Math., 23 (2012), pp.~1250028, 14.

\bibitem{Papelito2}
{\sc U.~Bruzzo and W.~D. Montoya}, {\em Codimension bounds for the
  {N}oether-{L}efschetz components for toric varieties}, Eur. J. Math.,
  (2021).
\newblock doi = 10.1007/s40879-021-00461-0, 9 pages.

\bibitem{Papelito3}
\leavevmode\vrule height 2pt depth -1.6pt width 23pt, {\em On the {H}odge
  conjecture for quasi-smooth intersections in toric varieties}, S\~{a}o Paulo
  J. Math. Sci., 15 (2021), pp.~682--694.

\bibitem{Papelito1}
\leavevmode\vrule height 2pt depth -1.6pt width 23pt, {\em An asymptotic
  description of the {N}oether-{L}efschetz components in toric varieties}.
\newblock {\tt arXiv:1905.01570v3}, 2022.

\bibitem{CarlsonGreenGriffithsHarris}
{\sc J.~Carlson, M.~Green, P.~A. Griffiths, and J.~Harris}, {\em Infinitesimal
  variations of {H}odge structure {(I)}}, Compositio Mathematica, 50 (1983),
  pp.~109--205.

\bibitem{ccd}
{\sc E.~Cattani, D.~Cox, and A.~Dickenstein}, {\em Residues in toric
  varieties}, Compositio Mathematica, 108 (1997), pp.~35--76.

\bibitem{CoxHom}
{\sc D.~A. Cox}, {\em The homogeneous coordinate ring of a toric variety}, J.
  Algebraic Geom., 4 (1995), pp.~17--50.

\bibitem{cox-res}
\leavevmode\vrule height 2pt depth -1.6pt width 23pt, {\em {Toric residues}},
  Arkiv f\"or Matematik, 34 (1996), pp.~73 -- 96.

\bibitem{Dan}
{\sc A.~Dan}, {\em Noether-{L}efschetz locus and a special case of the
  variational {H}odge conjecture: Using elementary techniques}, Aryasomayajula
  A., Biswas I., Morye A.S., Parameswaran A.J. (eds) Analytic and Algebraic
  Geometry. Springer, Singapore,  (2017).

\bibitem{GriffithsHarris}
{\sc P.~A. Griffiths and J.~Harris}, {\em Infinitesimal variations of {H}odge
  structure {(II)} : an infinitesimal invariant of {H}odge classes}, Compositio
  Mathematica, 50 (1983), pp.~207--265.

\bibitem{Ania}
{\sc A.~Otwinowska}, {\em Composantes de petite codimension du lieu de
  {N}oether-{L}efschetz: un argument asymptotique en faveur de la conjecture de
  {H}odge pour les hypersurfaces}, J. Algebraic Geom., 12 (2003), pp.~307--320.

\end{thebibliography}

\end{document}